\documentclass{amsart}

\usepackage{datetime}
\usepackage{amsfonts}
\usepackage{amsmath}
\usepackage{amssymb}
\usepackage{amsthm}
\usepackage{enumerate}
\usepackage{eucal}
\usepackage{graphicx,color}
\usepackage[usenames,dvipsnames,svgnames,table]{xcolor}

\usepackage{hyperref}
\hypersetup{colorlinks,
            filecolor=black,
            linkcolor=MidnightBlue,
            citecolor=NavyBlue,
            urlcolor=RoyalBlue,
            bookmarksopen=true}

\theoremstyle{plain}
\newtheorem{lemma}{Lemma}
\newtheorem{corollary}[lemma]{Corollary}

\newtheorem*{main}{Main Conjecture}

\newtheorem{remark}{Remark}

\theoremstyle{remark}

\newtheorem*{assumption}{Closeness Assumptions}

\newtheorem{case}{Case}

\theoremstyle{plain}
\newtheorem{conjecture}{Conjecture}

\newcommand{\ds}{\mr{d}s}

\newcommand{\mb}{\mathbb}
\newcommand{\mc}{\mathcal}

\newcommand{\mr}{\mathrm}
\newcommand{\po}{\big|_{\mc S^2_\pm}}

\newcommand{\ten}{\otimes}

\newcommand{\ve}{\varepsilon}
\newcommand{\vp}{\varphi}

\DeclareMathOperator{\Rc}{Rc}

\newcounter{mnotecount}[section]
\let\oldmarginpar\marginpar
\renewcommand\marginpar[1]{\-\oldmarginpar[\raggedleft\footnotesize #1]
{\raggedright\footnotesize #1}}

\begin{document}

\title[Asymptotic approach to K\"ahler singularity models]
{Non-K\"ahler Ricci flow singularities modeled on K\"ahler--Ricci  solitons\\~\\
\tiny{A chapter contributed to a \emph{Pure and Applied Mathematics Quarterly} volume
honoring Robert Bartnik's contributions to mathematical physics and mathematics}}

\author{James Isenberg}
\address[James Isenberg]{University of Oregon}
\email{isenberg@uoregon.edu}
\urladdr{http://www.uoregon.edu/$\sim$isenberg/}

\author{Dan Knopf}
\address[Dan Knopf]{University of Texas at Austin}
\email{danknopf@math.utexas.edu}
\urladdr{http://www.ma.utexas.edu/users/danknopf}

\author{Nata\v sa \v Se\v sum}
\address[Nata\v sa \v Se\v sum]{Rutgers University}
\email{natasas@math.rutgers.edu}
\urladdr{http://www.math.rutgers.edu/$\sim$natasas/}

\thanks{
JI thanks the NSF for support in PHY-1306441.
DK thanks the NSF for support in DMS-1205270.
N\v S thanks the NSF for support in DMS-0905749 and DMS-1056387.}

\begin{abstract}
We investigate Riemannian (non-K\"ahler) Ricci flow solutions that develop finite-time Type-I singularities and present
evidence in favor of a conjecture that parabolic rescalings at the singularities converge to singularity
models that are shrinking K\"ahler--Ricci solitons. Specifically, the singularity model for these solutions is expected to be the ``blowdown soliton" discovered in~\cite{FIK03}. Our partial results support the conjecture that the blowdown soliton is stable under Ricci flow, as well as the conjectured stability of the subspace of K\"ahler metrics under Ricci flow.
\end{abstract}

\maketitle

\section{Introduction}		\label{Intro}

While the behavior of Ricci flow is fairly well-understood  for three-dimensional Riemannian geometries, significantly less is known about four-dimensional Ricci flow. In this work, we study Ricci flow for a certain family of four-dimensional geometries (defined in Section 1.3) that develop finite-time Type-I singularities. Our interest in these geometries is to illuminate two outstanding issues concerning four-dimensional Ricci flows: i) the stability of certain singularity models in such flows, and ii) the behavior of Ricci flows that start at non-K\"ahler Riemannian geometries which are nonetheless close to K\"ahler geometries. To motivate our work here, we discuss each of these issues in turn. 

\subsection{Behavior of ``generic Ricci flow''}

One of the keys to understanding the nature of singularities that develop in solutions of $n$-dimensional Ricci flow is to adequately
classify the set of singularity models that may arise. Singularity formation in $3$-dimensional Ricci flow has been fairly well-understood since the work of Hamilton \cite{Ha1} and of Perelman \cite{Pe}. Indeed, it follows from the pinching estimate derived by
Ivey~\cite{Ivey93} and improved by Hamilton~\cite{Ha1} that the only possible $3$-dimensional
singularity models have nonnegative sectional curvature, which is a highly restrictive condition. By contrast,
M\'aximo's results~\cite{Maximo14} imply that, starting in dimension $n=4$, models of finite-time singularity
formation can have Ricci curvature of mixed sign (even for K\"ahler solutions).
As is well known, singularity models in every dimension have nonnegative scalar curvature.
However, as the only proven restriction on singularity models in dimensions $n\geq4$, this condition is too weak
to be very useful.

In dimensions $n\geq4$, therefore, a classification of all singularity models is impractical. A more promising alternative
is to try to classify those models that are generic, or at least stable. A singularity model developing from certain original data
is labeled \emph{stable} if flows starting from all sufficiently small perturbations of that data develop singularities with the same
singularity model; it is labeled \emph{generic} if flows that start from an open dense subset of all possible initial data develop singularities
having the same singularity model. Clearly, a singularity model can be generic only if it is stable.



Important work of Colding and Minicozzi (see~\cite{CM12} and \cite{CM15}) provides
strong support in favor of the conjecture that the only generic singularities of Mean Curvature Flow are
generalized cylinders $\mb R^m\times\mc S^{n-m}$. Although no analogous result is currently known
for Ricci flow, a conjectural picture comes from the work of  
Cao, Hamilton, and Ilmanen~\cite{CHI04}, who define the \emph{central density} $\Theta$ and the \emph{entropy}  $\nu(\mc M)$
of a shrinking Ricci soliton $\mc M$,  using Perelman's reduced volume and entropy, respectively (see~\cite{Pe}).
They observe that their central density imposes a partial order on shrinking solitons: monotonicity of the $\nu$-functional
in time means that if perturbations of a shrinking soliton develop singularities, these cannot be modeled on solitons
of lower density. (Compare~\cite{CM12}.) 

Motivated partly by \cite{CHI04}, it is conjectured by experts (see, e.g., \cite{HHS14}) that the only generic
singularity models in real dimension $n=4$ are $\mc S^4$, $\mc S^3\times\mb R$, $\mc S^2\times\mb R^2$
(all with their canonical metrics), and $(\mc L^2_{-1},h)$.\footnote{One reason this expectation is conjectural is that
it is not known if there exist Type-II singularity models on which Perelman's $\nu$-functional is undefined. However, 
such models, if they exist, are not expected to be generic.}
The manifold $(\mc L^2_{-1},h)$, which is constructed  and studied
in~\cite{FIK03}, is a $\mr U(2)$-invariant gradient K\"ahler shrinking soliton
on the complex line bundle\footnote{The bundle we label $\mc L^2_{-1}$ here is denoted by  $\mc L^{-1}$ in~\cite{FIK03}
and by $L(2,-1)$ in \cite{CHI04}.} $\mb C\hookrightarrow\mc L_{-1}^2\twoheadrightarrow\mb{CP}^1$, which is the
complex bundle $\mc O(-1)$;  \emph{i.e.,} it is the blow-up of $\mb C^2$ at the origin.
The next manifold on the list in~\cite{CHI04}, ordered by the central density $\Theta$, is $\mb{CP}^2$ with its
Fubini--Study metric. 

As noted above,  a pre-condition for a singularity model being generic is that it must be a stable attractor 
for Ricci flow --- regarded as a dynamical system on the space of Riemannian metrics. Stability of $\mc S^4$ is well established.
(In fact, Brendle and Schoen~\cite{BS} show that its basin of attraction includes all $1/4$-pinched metrics.)
Stability of generalized cylinders is strongly conjectured but not known for Ricci flow.
Stability of the blowdown soliton is also not known, although M\'aximo's proof~\cite{Maximo14} shows that arbitrarily small
$\mr U(2)$-invariant K\"ahler perturbations of the unstable shrinking soliton on $\mb{CP}^2\#\,\overline{\mb{CP}}^{\,2}$
(which was discovered independently by Koiso~\cite{Koiso90} and by Cao~\cite{Cao96})
develop singularities modeled on $(\mc L^2_{-1},h)$. (We remark that prior to M\'aximo's results, it was shown in ~\cite{HM11} that 
the Koiso-Cao soliton is \emph{linearly} unstable. We note also that M\'aximo's results were extended to general dimensions by 
Guo and Song~\cite{GS17}, who thus establish in full generality the conjecture made in part~(3) of Example~2.2
in~\cite{FIK03}.) $\mb{CP}^2$ is well known to be weakly variationally stable, and was expected by many to be stable.
However, Kr\"oncke~\cite{Kro13} has proven that it is dynamically unstable. (This has recently been independently
verified by two of the authors~\cite{KS17}.) That leaves $(\mc L^2_{-1},h)$ as a critical ``borderline'' case.   Our results
in this chapter provide some evidence in favor of the conjectured dynamic stability of $(\mc L^2_{-1},h)$. If true,
this would indicate an incomplete analogy between Ricci flow and mean curvature flow, where only generalized
cylinders are stable~\cite{CM12}.
\smallskip

While the construction of the $(\mc L^2_{-1},h)$ shrinker involves the blowup of a point on $\mb C^2$, following  the authors of ~\cite{CHI04}), we call $(\mc L^2_{-1},h)$ the \emph{blowdown soliton}.  We do this because, as shown in Theorem~1.6 of~\cite{FIK03},
there is a family of Riemannian manifolds $\mc N_t$, $-\infty<t<\infty$, with the following features: for  $t<0$, $\mc N_t$
is $(\mc L^2_{-1},h(t))$; for $t=0$, $\mc N_0$ is a K\"ahler cone on $\mb C^2$ with an isolated singularity at the origin;
and for $t>0$, $\mc N_t$ is an expanding soliton discovered by Cao~\cite{Cao97}.
It is expected that according to most (if not all) of the definitions of a \emph{weak solution of Ricci flow} which are currently
being explored (e.g., see \cite{HN15} and \cite{Stu16}), the family $\mc N_t$ will qualify for such a designation. 
Consequently, in this weak sense, one sees that Ricci flow can carry out
a blowdown, understood in the sense of algebraic geometry.
\smallskip

Our results in this chapter provide significant, albeit incomplete, evidence that the blowdown soliton is a singularity model attractor for solutions of Ricci flow that originate from a set of compact Riemannian initial data defined by a structural (isometry) hypothesis and by a (weak) set of  pinching 
conditions that we specify in Section~\ref{sec:evolve} below. Metrics in this set are not K\"ahler. As noted above,
these partial results provide some evidence in favor of the conjectured stability of $(\mc L^2_{-1},h)$. What prevents
this chapter from providing a complete proof is its reliance on two technical conjectures discussed below, for which we
present formal arguments but thus far lack rigorous arguments.

\subsection{Behavior of Ricci flow near K\"ahler geometries}

As noted above, the $(\mc L^2_{-1},h)$ shrinker is K\"ahler. Hence, the study of non-K\"ahler Ricci flows near the blowdown 
soliton provides information about the difficult issue of  the behavior of Ricci flow solutions that start near, but not in, the 
subspace of K\"ahler metrics. Do those solutions
stay near or (better) asymptotically approach that subspace, which is of infinite codimension? It is believed by many experts that the subspace of K\"ahler metrics \emph{should} be dynamically stable for nearby solutions of Ricci flow. Evidence of favor of this conjecture is provided by the work of Streets and Tian~\cite{ST11}, who prove that the K\"ahler subspace is an attractor for Hermitian curvature flow.

While our results fall far short of  a general stability principle for K\"ahler geometries, the partial results and a formal
argument presented later in this work do
support a conjectural picture of  non-K\"ahler solutions of Ricci flow that become asymptotically K\"ahler, in suitable
space-time neighborhoods of developing singularities, at rates that break scaling invariance. We hope that the evidence
we give here provides motivation for  further study of this general question, particularly in (real) dimension $n=4$.

\subsection{Organization}

The general class of Riemannian geometries  that we study in this work are smooth cohomogeneity-one metrics on the closed manifold $\mc S^2\tilde\times\mc S^2$ (the ``twisted bundle" of $\mc S^2$ over $\mc S^2$). We describe these geometries (which we label ``[$\mc S^2\tilde\times\mc S^2$]-warped Berger geometries") in detail below in Section \ref{sec:evolve}. Here, for the purposes of stating our main conjecture,
we note that for these metrics, there are two distinguished fibers $\mc S^2_\pm$ (at either ``pole''); by contrast, 
a generic fiber is diffeomorphic to $\mc S^3$.


In Section~\ref{sec-assume}, we identify an open subset of the [$\mc S^2\tilde\times\mc S^2$]-warped Berger geometries by means of five pinching inequalities.  These inequalities constitute our Closeness Assumptions, which we require the initial data for our Ricci flow solutions to satisfy. These assumptions ensure that our initial data, while not K\"ahler, are ``not too far'' from the subspace of K\"ahler metrics. In Section~\ref{sec-initial}, we prove that our assumptions are not vacuous;
\emph{i.e.,} we show that the open subset of initial data satisfying the Closeness Assumptions is not empty.

We clarify the relationship between K\"ahler geometries and the [$\mc S^2\tilde\times\mc S^2$]-warped Berger geometries in
Section~\ref{SingularityModel}. Also in that section, we provide some background information about the blowdown soliton.
\smallskip
 
In the remainder of this work, we prove a sequence of Lemmata and Corollaries that  combine to yield an
\emph{almost complete} proof --- modulo two technical conjectures discussed below --- of  the following result:

\begin{main}
There exists a nonempty open set of non-K\"ahler metrics on $\mc S^2\tilde\times\mc S^2$ 
(contained in the [$\mc S^2\tilde\times\mc S^2$]-warped Berger class, and satisfying the Closeness Assumptions)
such that any Ricci flow solution originating from this set has the following properties:
\begin{enumerate}
\item Inequalities (a)--(d) in the Closeness Assumptions are preserved by the flow.
\item The solution develops a Type-I singularity at $T<\infty$, with
 $|\mc S^2_-(T)|=0$.\footnote{We denote the area of either exceptional fiber at any time $t\in[0,T]$ by
$|\mc S^2_\pm(t)|$.}

\item Every blow-up sequence $\big(\mc S^2\tilde\times\mc S^2,G_k(t),p\big)$ with $p\in\mc S^2_-$
subconverges to a K\"ahler singularity model that is the blowdown shrinking soliton $(\mc L^2_{-1},h)$.
\end{enumerate}
\end{main}

\textbf{Acknowledgment} The authors thank Alexander Appleton for discovering a mistake in the original
version of this work.

\section{The set-up}		\label{sec:evolve}

\subsection{Topology and geometry}
In~\cite{IKS14}, we study  ``warped Berger'' metrics which take the form
\begin{equation}	\label{metric-standard}
 G = \ds\ten\ds+\Big\{f^2\,\omega^1\ten\omega^1
 	   + g^2\big(\omega^2\ten\omega^2
	   + \omega^3\ten\omega^3\big)\Big\}
\end{equation}
on $[s_-,s_+]\times\mr{SU}(2)$, where $\{ \omega^1, \omega^2, \omega^3\}$ constitutes a one-form basis for $\mr{SU}(2)$, where  $s(x,t)$ denotes arclength from $x=0$, with $x\in[-1,1]$, and where we set  
$s_\pm:=s(\pm1)$. The functions $f$ and $g$ depend only on $x$ (or equivalently on $s$); hence these metrics are cohomogeneity one. In~\cite{IKS14}, we choose boundary conditions on $f$ and $g$ that result in these metrics inducing geometries on $\mc S^3\times\mc S^1$. Here, we instead choose boundary conditions on $f$ and $g$ that result in smooth cohomogeneity geometries  on $\mc S^2\tilde\times\mc S^2$, thereby defining the class of [$\mc S^2\tilde\times\mc S^2$]-warped Berger geometries. We do this as follows.

It is a standard result in Riemannian geometry that one may smoothly
close the boundary at $s_-$, provided that the functions
$f_-(s):=f(s_-+s)$ and $g_-(s):=g(s_-+s)$ defined for $0\leq s\leq s_+-s_-$ satisfy
\begin{equation}	\label{close-}
f_-^{(\mr{even})}(0)=0,\quad
f_-'(0)=1,\qquad\mbox{and}\qquad
g_-(0)>0,\quad g_-^{(\mr{odd})}(0)=0.
\end{equation}
The topology then locally becomes that of the disc bundle
$\mc D^2\hookrightarrow\mc D_1^4\twoheadrightarrow\mc S^2$
with Euler class $1$ and boundary $\partial\mc D_1^4\approx\mc S^3$
that appears in the handlebody construction of $\mb{CP}^2$. Note that the
$2$-sphere here is the base of the Hopf fibration on $\mc S^3\approx\mr{SU}(2)$.
If one repeats this construction at $s_+$, with
$f_+(s):=f(s_++s)$ and $g_+(s):=g(s_++s)$ defined for $s_--s_+\leq s\leq 0$ satisfying
\begin{equation}	\label{close+}
f_+^{(\mathrm{even})}(0)=0,\quad
f_+'(0)=-1,\qquad\mbox{and}\qquad
g_+(0)>0,\quad g_+^{(\mr{odd})}(0)=0,
\end{equation}
one obtains a closed $4$-manifold with the topology of
$\mc S^2\tilde\times\mc S^2$.
We denote by $\mc S^2_\pm$ the distinguished $2$-spheres that
appear as the fibers in the closing construction at either ``pole'' $s_\pm$.
We note that while $\mc S^2\tilde\times\mc S^2$ is diffeomorphic to $\mb{CP}^2\#\,\overline{\mb{CP}}^{\,2}$,
the Ricci flow evolutions we study are not K\"ahler.

The metrics $G = \ds\ten\ds+f^2\,\omega^1\ten\omega^1+g^2\omega^2\ten\omega^2%
+h^2\omega^3\ten\omega^3$ described in Appendix~A of \cite{IKS14} are clearly
$\mr{SU}(2)$-invariant. The simplifying assumption $h\equiv g$ made here enlarges their
symmetry group to $\mr U(2)$. However, although $\mb{CP}^2\#\,\overline{\mb{CP}}^{\,2}$ admits
K\"ahler metrics, including the $\mr U(2)$-invariant K\"ahler--Ricci soliton mentioned above,
we observe in Lemma~\ref{CalabiCondition} that metrics of the form~\eqref{metric-standard} cannot be
K\"ahler unless they satisfy the closed condition $f=gg_s$.

\subsection{Ricci flow equations}
In this section, we investigate solutions $\big(\mc S^2\tilde\times\mc S^2,G(t)\big)$ of Ricci flow that originate  from smooth initial data $G(0)$ satisfying the closing conditions~\eqref{close-} and~\eqref{close+} for [$\mc S^2\tilde\times\mc S^2$]-warped Berger geometries, as outlined above. For as long as such solutions remain smooth, the functions  $f$ and $g$ continue to satisfy conditions
\eqref{close-} and \eqref{close+}, and hence remain  [$\mc S^2\tilde\times\mc S^2$]-warped Berger geometries.

Since the metrics studied in  \cite{IKS14} and those studied here are the same apart from boundary conditions,
we may use formulas~(10)--(13) of  \cite{IKS14} to obtain the sectional curvatures\footnote{Using
L'H\^opital's rule, it is straightforward to verify that all quantities appearing in this section are well
defined at $\mc S^2_\pm$. We make this explicit below.}   of the metric $G$:
\begin{subequations}		\label{curvatures}	
\begin{align}
\kappa_{12}=\kappa_{31}&=\frac{f^2}{g^4}-\frac{f_s g_s}{fg},\\
\kappa_{23}&=\frac{4g^2-3f^2}{g^4}-\frac{g_s^2}{g^2},\\
\kappa_{01}&=-\frac{f_{ss}}{f},\\
\kappa_{02}=\kappa_{03}&=-\frac{g_{ss}}{g}.
\end{align}
\end{subequations}

Writing the metric in coordinate form~\eqref{metric-standard},
we note that its evolution under Ricci flow is governed by the evolution
equations for $f$ and $g$, which (as shown in~(14) of \cite{IKS14})
take the following form:
\begin{subequations}		\label{RF}
\begin{align}
	\label{f-evolution}
f_t&=f_{ss}+2\frac{g_s}{g}f_s-2\frac{f^3}{g^4},\\
	\label{g-evolution}
g_t&=g_{ss}+\left(\frac{f_s}{f}+\frac{g_s}{g}\right)g_s
+2\frac{f^2-2g^2}{g^3}.
\end{align}
\end{subequations}
 The variable $s=s(x,t)$, representing arclength from the $\mc S^3$ at $x=0$, is a choice of
gauge that results in this system being manifestly strictly parabolic. The cost one pays for this is the non-vanishing commutator,
\begin{equation}	\label{Commutator}
    \left[\frac{\partial}{\partial t},\frac{\partial}{\partial s}\right]
    =-\left(\frac{f_{ss}}{f}+2\frac{g_{ss}}{g}\right)\frac{\partial}{\partial s}.
\end{equation}

\subsection{Closeness Assumptions}		\label{sec-assume}
The Riemannian Ricci flow solutions we study originate from an open set of cohomogeneity-one metrics
that is  defined by certain mild hypotheses,
which effectively guarantee that at least initially, the metrics are  ``somewhat close'' to the subspace of K\"ahler metrics.
\begin{assumption} \label{assumption-main}
At time $t=0$, the metric $G$ of the form~\eqref{metric-standard} determined by the pair $(f,g)$ satisfies the following:
\begin{enumerate}
\item[(a)] $f\leq g$;
\item[(b)] $gg_s\leq f$;
\item[(c)] $|f_s| \le 2/\sqrt3$;
\item[(d)] $g^2(s_+)-3g^2(s_-)\geq\delta^2$ for some $\delta>0$;
\item[(e)] $g_s\geq0$, with strict inequality off $\mc S^2_\pm$.
\end{enumerate}
\end{assumption}
 
It follows from  Lemma~26 of \cite{IKS14} that condition (a) is preserved under the flow. We prove in Section~\ref{FirstOrder} that
condition~(b) --- which, as we show there, may be regarded as a ``K\"ahler pinching condition'' --- and condition~(c)
are preserved by the flow. We prove in Section~\ref{Singular} that (d) is preserved. We explain the motivation for condition~(e), which we do not prove is preserved, in the discussion of Conjecture~\ref{SouthPole}.

\begin{remark}
Even for K\"ahler--Ricci flow solutions, condition (d)~is necessary
for the $g^2\big(\omega^2\ten\omega^2 + \omega^3\ten\omega^3\big)$ factor to vanish before the $\big(\ds\ten\ds+f^2\,\omega^1\ten\omega^1\big)$ factor does. This is necessary for the development of a local singularity on $\mc S^2_-$
(see Theorem~1.1 of \cite{SW11} and Remark~\ref{kappa-dichotomy} below).
\end{remark}

\subsection{Construction of metrics satisfying the Closeness Assumptions}		\label{sec-initial}
We choose  $f$ to be any smooth function that is defined for $s\in[s_-,s_+]$, is strictly positive except at $s_\pm$, satisfies $|f_s|\leq1$ with equality only at $s_\pm$, and satisfies the closing conditions~\eqref{close-} and~\eqref{close+}.
For each such function, we now construct an infinite-dimensional family $\mc G_{\alpha,\delta,\ve}$ of initial metrics which 
satisfy our Closeness Assumptions. The family depends on parameters $\alpha,\delta$, and $\ve$, to be chosen below. We define
\[
A^2:=2\int_{s_-}^{s_+} f(s)\,\mr d s,
\]
noting that we are free to let the difference $s_+-s_-$, and hence $A^2$, be as large as we wish.
We then choose $\alpha$ and $\delta$ to be any positive parameters satisfying
\begin{equation}	\label{2d}
\alpha^2+\delta^2\leq\frac{A^2}{2}.
\end{equation}


To define $g$, and hence a metric $(f,g)\in\mc G_{\alpha,\delta,\ve}$, 
we choose $\vp$ to be  any smooth function satisfying $1-\ve\leq\vp\leq1$, requiring that it be nonconstant
unless $\ve=0$. Clearly $\ve$ controls how much  $\vp$ can stray from  being constant.
We then set 
\[
g^2(s):=\alpha^2+2\int_{s_-}^s \vp(\bar s)f(\bar s)\,\mr d\bar s.
\]
We readily verify  that $g$ defined in this way satisfies closing conditions~\eqref{close-} and~\eqref{close+}.

To verify that part~(a) of our Closeness Assumptions is satisfied, we observe that the gradient
restriction $|f_s|\leq1$ implies that
\begin{align*}
f^2(s)&=2 \int_{s_-}^s f(\bar s)f_{\bar s}(\bar s)\,\mr d\bar s\\
	&\leq2 \int_{s_-}^s \big\{\vp+(1-\vp)\big\}f(\bar s)\,\mr d\bar s\\
	&\leq g^2-\alpha^2+\ve A^2\\
	&\leq g^2,
\end{align*}
provided that
\begin{equation}
\ve\leq\frac{\alpha^2}{A^2}.
\end{equation}
It immediately follows that part~(b) holds for all $\ve\in[0,1)$, with equality --- which  Lemma~\ref{CalabiCondition}
(below) shows is equivalent to the metric being K\"ahler --- if and only if $\ve=0$.
Part~(c) holds as a consequence of  the gradient restriction $|f_s|\leq1$.

To verify that part~(d) holds, we observe that one has
\begin{align*}
g^2(s_+)-3g^2(s_-)&\geq\big\{\alpha^2+(1-\ve)A^2\big\}-3\alpha^2\\
&=(A^2-2\alpha^2)-\ve A^2\\
&\geq2\delta^2-\ve A^2,
\end{align*}
with the last inequality following from the restrictions on $\alpha$ and $\delta$ that we have imposed  in~\eqref{2d}. Hence we satisfy part~(d) so long as
\begin{equation}
\ve\leq\frac{\delta^2}{A^2}.
\end{equation}

Because $gg_s\geq(1-\ve)f\geq0$, it is clear that part~(e) is satisfied.

\section{Characterizing K\"ahler metrics}	\label{SingularityModel}
In this work, we provide evidence that as they become singular, solutions originating from initial data satisfying our Closeness Assumptions
asymptotically approach the blowdown soliton.  For the reader's convenience, we include here a brief review of metrics related
to that singularity model.

\subsection{The Calabi construction}		\label{CalabiConstruction}
We call a metric on $\mb{CP}^2\#\,\overline{\mb{CP}}^{\,2}$ or $\mc L_{-1}^2$ a
\emph{Calabi metric}  if it is both K\"ahler and $\mr U(2)$-invariant.
As part of a much more general construction \cite{Calabi82}, Calabi has observed that
any $\mr U(2)$-invariant K\"ahler metric on $\mb C^2\backslash(0,0)$ has the form
\begin{equation}	\label{U2-complex}
h_{\,\mb C^2\backslash(0,0)}
=\Big\{e^{-r}\phi\,\delta_{\alpha\beta}
+e^{-2r}(\phi_r-\phi)\,\bar z^\alpha z^\beta\Big\}\,
	\mr dz^\alpha\ten\mr d\bar z^\beta.
\end{equation}
Here  $r:=\log(|z_1|^2+|z_2|^2)$ is Calabi's coordinate, and $\phi(r)=P_r(r)$, where $P$ is the
K\"ahler potential. The metric closes smoothly at the origin, hence induces
a smooth metric on the total space of the bundle $\mc L_{-1}^2$ (or on a neighborhood
of $\mc S^2_-$ in $\mb{CP}^2\#\,\overline{\mb{CP}}^{\,2}$) if and only if there are
$a_0,a_1>0$ such that
\begin{equation}	\label{ClosingCondition}
\phi(r)=a_0+a_1 e^r + a_2 e^{2r}+ \mc O(e^{3r})\quad\mbox{ as }\quad
r\rightarrow-\infty.
\end{equation}
The metric closes smoothly at spatial infinity, hence induces a smooth K\"ahler
metric with respect to the unique complex structure on $\mb{CP}^2\#\,\overline{\mb{CP}}^{\,2}$,
if and only if two conditions hold: i) $\phi_r>0$ everywhere, and ii) there are $b_0>a_0$ and $b_1<0$ such that
\[
\phi(r)=b_0+b_1 e^{-r} + b_2 e^{-2r}+ \mc O(e^{-3r})\quad\mbox{ as }\quad
r\rightarrow\infty.
\]
Alternatively, one may obtain a complete Calabi metric on the noncompact space
$\mc L^2_{-1}$ by imposing conditions at spatial infinity that guarantee completeness;
see, \emph{e.g.,} \cite{FIK03}.
As noted in equation~(19) of that paper, any $\mr U(2)$-invariant metric on
$\mb C^2\backslash(0,0)$ can be written in real coordinates on $\mb R^4\backslash(0,0,0,0)$ as
\begin{equation}		\label{U2-real}
h_{\,\mb R^4\backslash(0,0,0,0)}
=\phi_r\big(\frac14\mr dr\ten\mr dr+\omega^1\ten\omega^1\big)
+\phi\big(\omega^2\ten\omega^2+\omega^3\ten\omega^3\big).
\end{equation}

\subsection{A coordinate transformation}	\label{Diffeos}

A comparison of equations~\eqref{metric-standard} and \eqref{U2-real} shows that a coordinate
transformation is needed to write a Calabi metric in the $s$-coordinate system. We implement
this as follows. Recalling that $s(x,t)$ denotes arclength from the $\mc S^3$ at the ``interior'' point $x=0$, and
motivated by Calabi's (fixed) $r$-coordinate introduced in Section~\ref{CalabiConstruction},
we define here a function
\begin{equation}	\label{magic}
\varrho(s,t):=2\int_0^s\frac{\mr d\bar s}{f(\bar s,t)}.
\end{equation}
The  closing conditions then show that $\varrho\rightarrow\pm\infty$ at $\mc S^2_\pm$. Moreover, one has
\begin{equation}	\label{dsdr}
\ds=\frac12f\,\mr d\varrho,
\end{equation}
so that equation~\eqref{metric-standard} may be re-expressed in the form
\begin{equation}	\label{transformation}
 G = f^2\big(\frac14 \mr d\varrho\ten\mr d\varrho+\omega^1\ten\omega^1\big)
 	   + g^2\big(\omega^2\ten\omega^2+ \omega^3\ten\omega^3\big),
\end{equation}
where we emphasize that the coordinate $\varrho$ is allowed to depend on time.
We note that $\varrho$ and its time evolution depend only on $s(x,t)$ and $f(s(x,t),t)$,
neither of which depend on $\varrho$.

We observe that equation~\eqref{transformation} has the form~\eqref{U2-real} of a
Calabi metric $h$ if and only if $g^2=\phi$ and $f^2=(g^2)_\varrho$, in which case one has
\[
f=gg_s\qquad\mbox{ and }\qquad f_s=gg_{ss}+g_s^2.
\]

We summarize this simple observation, which is crucial to our work here, as follows:

\begin{lemma}	\label{CalabiCondition}
A $[\mc S^2\tilde\times\mc S^2]$-warped Berger metric~\eqref{metric-standard} is K\"ahler if and only if $f=gg_s$.
\end{lemma}

If $G$ is K\"ahler, then its sectional curvatures, which generally take the form~\eqref{curvatures}, take the following special form:
\begin{align*}
\kappa_{12}=\kappa_{31}=\kappa_{02}=\kappa_{03}&=-\frac{g_{ss}}{g},\\
\kappa_{23}&=4\frac{1-g_s^2}{g^2},\\
\kappa_{01}&=-\frac{g_{sss}}{g_s}-3\frac{g_{ss}}{g}.
\end{align*}
As must be true for a K\"ahler metric on a complex surface, the Ricci endomorphism
then has only two eigenvalues,
\[
R_0^0=R_1^1=-\frac{g_{sss}}{g_s}-5\frac{g_{ss}}{g}\qquad\mbox{ and }\qquad
R_2^2=R_3^3=-2\frac{g_{ss}}{g}+4\frac{1-g_s^2}{g^2}.
\]
\smallskip

Because K\"ahler--Ricci flow is strictly parabolic, no time-dependent choice of gauge $s(x,t)$ is needed to ensure parabolicity.
Rather, one can write the K\"ahler--Ricci \textsc{pde} with respect to a time-independent coordinate. The following observation is a particular instance of this general fact.

\begin{lemma}	\label{rho-evolution-proof}
The evolution equation for the coordinate $\varrho$ under Ricci flow takes the form
\begin{equation}	\label{rho-evolution}
\varrho_t=2\int_0^\varrho\left\{\frac{g_{ss}}{g}-\frac{f_sg_s}{fg}+\frac{f^2}{g^4}\right\}\mr d\bar\varrho.
\end{equation}
For a K\"ahler geometry, the integrand in~\eqref{rho-evolution} vanishes pointwise; hence, 
for K\"ahler initial data, the coordinate $\varrho$ is independent of $t$.
\end{lemma}

\begin{remark}
For K\"ahler initial data, one may therefore assume without loss of generality that $\varrho$ is identical to Calabi's
coordinate $r=\log(|z_1|^2+|z_2|^2)$.
\end{remark}

\begin{proof}[Proof of Lemma~\ref{rho-evolution-proof}]
It follows from  equation~(56) in \cite{IKS14} and from ~\eqref{curvatures} above that 
the gauge quantity $\frac{\partial s}{\partial x}$ evolves according to the equation
\[
\Big(\frac{\partial s}{\partial x}\Big)_t=
\left\{\frac{f_{ss}}{f}+2\frac{g_{ss}}{g}\right\}\frac{\partial s}{\partial x}.
\]
Hence, using equations~\eqref{magic} and~\eqref{f-evolution}, we determine that the time derivative of $\varrho$
at fixed $x$ is given by
\begin{align*}
\frac12\varrho_t&=\frac{\partial}{\partial t}
\left(\int_0^x f^{-1}\big(s(\bar x,t),t\big)\,\frac{\partial s}{\partial\bar x}\,\mr d\bar x\right)\\
&=\int_0^x\left\{f^{-1}\Big(\frac{\partial s}{\partial\bar x}\Big)_t-f^{-2}f_t\Big(\frac{\partial s}{\partial\bar x}\Big)
\right\}\mr d\bar x\\
&=2\int_0^s\left\{\frac{g_{\bar s\bar s}}{fg}-\frac{f_{\bar s}g_{\bar s}}{f^2g}+\frac{f}{g^4}\right\}\mr d\bar s.
\end{align*}
This proves the first claim. The second follows by direct computation.
\end{proof}

\subsection{Ricci flow of Calabi metrics}
Lemma~\ref{CalabiCondition} states that the initial metric is Calabi if and only if $f=gg_s$.
Because Ricci flow preserves the K\"ahler condition with respect to the original complex structure
(here, the unique complex structure on $\mb{CP}^2\#\,\overline{\mb{CP}}^{\,2}$) and also preserves
initial symmetries, a solution originating from Calabi initial data remains Calabi for as long as it exists.
This can be seen directly, as we now observe.

In this section (which is not needed for the rest of the chapter) and occasionally below,
we find it convenient to work with $u:=f^2$ and $v:=g^2$. The Ricci flow evolution equations for these quantities 
are given by
\begin{subequations}
\begin{align}
u_t &= u_{ss}-\frac{u_s^2}{2u}+\frac{u_s v_s}{v}-4\frac{u^2}{v^2}, \label {u-evolution} \\
v_t &= v_{ss}+\frac{u_s v_s}{2u}+4\frac{u-2v}{v}.	\label{v-evolution}
\end{align}
\end{subequations}
 On a Calabi solution, one can use the relation $u=v_s^2/4$ (equivalent to  $f=gg_s$) to simplify the evolution
equation above for $v$, thereby obtaining
\begin{equation}	\label{Calabi-v-evolution}
v_t=2v_{ss}+\frac{v_s^2}{v}-8.
\end{equation}
One now has two ways of computing the evolution of $u$. Evaluating the equation above for $u_t$ by
using the K\"ahler condition $u=v_s^2/4$  to convert the  \textsc{rhs} into terms
involving only $v$ and its derivatives, one obtains
\begin{equation}	\label{Calabi-u-evolution}
u_t=\frac12v_sv_{sss}+\frac12\frac{v_{ss}v_s^2}{v}-\frac14\frac{v_s^4}{v^2}.
\end{equation}
On the other hand, one can differentiate the \textsc{rhs} of  $u=v_s^2/4$ directly, use
the commutator $[\partial_t,\partial_s]$ given in equation \eqref{Commutator}, and then
apply \eqref{Calabi-v-evolution}, obtaining
\begin{align*}
\left(\frac{v_s^2}{4}\right)_t
	&=\frac12v_s(v_s)_t\\
	&=\frac12v_s\left\{(v_t)_s-\left(\frac{f_{ss}}{f}+2\frac{g_{ss}}{g}\right)v_s\right\}\\
	&=\frac12v_sv_{sss}+\frac12\frac{v_{ss}v_s^2}{v}-\frac14\frac{v_s^4}{v^2},
\end{align*}
as above. This calculation verifies directly what one knows from general principles: 
that the Calabi condition is preserved by the flow. We note in particular that for a
Calabi solution, the Ricci flow system reduces to a scalar \textsc{pde}, in the sense that the evolution
of $u$ is completely determined by the evolution of $v$.

\begin{remark}
For solutions with initial data satisfying our Closeness Assumptions, the fact that $g_s>0$ everywhere except at $\mc S^2_\pm$ holds initially. For as long as this remains true (possibly only a short time for non-K\"ahler solutions), there is a well-defined function $\theta$ such that
\[
f=\theta\,gg_s.
\]
We note that $\theta\equiv1$ for a K\"ahler solution, and that the evolution equation for $\theta$ is 
\[
\theta_t=\theta_{ss}+2\frac{fg_s-2f_sg}{g^2}(\theta^2-1),
\]
which yields an easy direct proof that the K\"ahler condition is preserved for these geometries.
\end{remark}

\subsection{The blowdown soliton}	\label{ModelSoliton}

Under K\"ahler--Ricci flow, the evolution of an arbitrary Calabi metric
\begin{equation}
\label{Calabih}
h =\Big\{e^{-r}\phi\,\delta_{\alpha\beta}
+e^{-2r}(\phi_r-\phi)\,\bar z^\alpha z^\beta\Big\}\,
	\mr dz^\alpha\ten\mr d\bar z^\beta,
\end{equation}
written in terms of Calabi's fixed $r$-coordinate on $\mc L_{-1}^2$ or on $\mb{CP}^2\#\,\overline{\mb{CP}}^{\,2}$, is determined by the  \textsc{pde}
\begin{equation}	\label{KRF}
\phi_t=\frac{\phi_{rr}}{\phi_r}+\frac{\phi_r}{\phi}-2.
\end{equation}

The blowdown soliton is specified by setting $\phi$ in \eqref{Calabih} equal to a function $\vp$ which (following 
Lemma~6.1 and equation~(27) of~\cite{FIK03} with $\lambda=-1$, $\mu=\sqrt2$, and $\nu=0$)
satisfies the separable first-order \textsc{ode}
\begin{equation}	\label{phi-ode1}
\vp_r=\frac{1}{\sqrt2}\vp-(\sqrt2-1)-\left(1-\frac{1}{\sqrt{2}}\right)\vp^{-1}.
\end{equation}
Rewriting this \textsc{ode} in the form
\[
\mr dr = \frac{\vp\,\mr d\vp}{\vp-1}-\frac{\vp\,\mr d\vp}{\vp+\sqrt2-1},
\]
one can solve it implicitly, obtaining
\begin{equation}	\label{phi-implicit}
e^{r+\chi}=\frac{\vp-1}{\big(\vp+\sqrt2-1\big)^{\sqrt2-1}}.
\end{equation}
The arbitrary constant $\chi$ above reflects the fact that the soliton is unique
only modulo translations in $r$. Examination of formula~\eqref{phi-implicit} also
shows that the soliton is cone-like at spatial infinity and hence complete.

Equation~(24) of \cite{FIK03} implies that the blowdown soliton function $\vp$ also satisfies the second-order \textsc{ode}
\begin{equation}	\label{phi-ode2}
\frac{\vp_{rr}}{\vp_r}+\frac{\vp_r}{\vp}-\sqrt2\vp_r+\vp-2=0.
\end{equation}
It follows from \eqref{phi-ode2} that $\vp$ evolves by
\begin{equation}	\label{BlowdownEvolution}
\vp_t=\sqrt2\vp_r-\vp
\end{equation}
In particular, the soliton evolves by translation and scaling.

\section{Basic estimates}

In this section, we prove several estimates that support our Main Conjecture. 

\subsection{A weak one-sided K\"ahler stability result}	\label{Weak}


We begin by introducing the useful quantity
\begin{equation}	\label{KahlerQuantity}
\psi:=\left(\frac{gg_s}{f}\right)^2-1.
\end{equation}
This quantity $\psi$ is well defined at $\mc S^2_\pm$, because it follows from  l'H\^opital's rule that $\frac{gg_s}{f}\po=gg_{ss}$.

Lemma~\ref{CalabiCondition} tells us that $\psi\equiv0$ if and only if the metric  $G$ from \eqref{metric-standard} is K\"ahler. 
Therefore,  we use $\psi$ to measure, in a precise sense, how far away a solution is from being K\"ahler.
The following result is thus a statement of weak (one-sided) stability for the K\"ahler condition.
Note that part~(b) of our Closeness Assumptions ensures that $\psi\leq0$ at $t=0$.

\begin{lemma}	\label{NearCalabi}
If $-1\leq\psi\leq0$ initially, then $-1\leq\psi\leq0$ as long as the flow exists.
\end{lemma}
\begin{proof}
It is only necessary to prove the upper bound. The quantity $\psi$ evolves under Ricci flow by
\begin{equation}	\label{psi-evolution}
\psi_t=\psi_{ss}+\left\{3\frac{f_s}{f}-2\frac{g_s}{g}\right\}\psi_s
-\frac{\psi_s^2}{2(\psi+1)}+\left\{4\frac{g_s^2}{g^2}-8\frac{f_sg_s}{fg}\right\}\psi.
\end{equation}
From this equation, it is clear that the condition $\psi\leq0$ is preserved if all
maxima of $\psi$ occur away from $\mc S^2_\pm$.

If a maximum occurs instead at $\mc S^2_\pm$, then we apply  l'H\^opital to determine  that
\[
\frac{f_s\psi_s}{f}\po=\psi_{ss}\qquad\mbox{and}\qquad
\frac{f_sg_s}{fg}\po=\frac{g_{ss}}{g}.
\]
Hence
\[
\psi_t\po=4\psi_{ss}-8\frac{g_{ss}}{g}\psi.
\]
However, smoothness of either function $\psi_\pm(s,\cdot):=\psi(s-s_\pm,\cdot)$ at a
maximum on $\mc S^2_\pm$ shows that $\psi_{ss}\po=(\psi_\pm)_{ss}\po\leq0$. The result follows.
\end{proof}

\subsection{First-derivative estimates}	\label{FirstOrder}

Based on the one-sided K\"ahler stability established in Lemma \ref{NearCalabi}, we now derive estimates on the first derivatives of $f$ and $g$, and consequently on the curvatures which depend on these first derivatives.
We first state an immediate corollary of Lemma \ref{NearCalabi}, which controls $|g_s|$.

\begin{corollary}
\label{cor-gs}
Solutions originating from initial data satisfying our Closeness Assumptions have $|g_s|\leq1$ for as long as they exist.
\end{corollary}

\begin{proof}
Because, as noted above, the ordering $f\leq g$ is preserved by the flow, it follows from
Lemma~\ref{NearCalabi} that $f^2g_s^2\leq g^2g_s^2\leq f^2$. The result follows.
\end{proof}

Next, we  obtain a bound for $|f_s|$.
\begin{lemma}
\label{lem-fs}
If $f\leq g$ initially, then for as long as the flow exists,
\[
|f_s|\leq\max\left\{\frac{2}{\sqrt3},\,\max|f_s(\cdot,0)|\right\}.
\]
\end{lemma}
\begin{proof}
Using equation~~(21) of \cite{IKS14} together with the fact that
\[\Delta\zeta=\zeta_{ss}+(f_s/f+2g_s/g)\zeta_s
\]
holds for any smooth function $\zeta(s,t)$, we see that $f_s$ evolves by
\begin{equation}	\label{f_s-evolution}
(f_s)_t=(f_s)_{ss}+\left\{2\frac{g_s}{g}-\frac{f_s}{f}\right\}(f_s)_s
 	-\left\{6\frac{f^2}{g^4}+2\frac{g_s^2}{g^2}\right\}f_s
	+8\frac{f^3}{g^5}\,g_s.
\end{equation}
Because $f_s\po=\pm1$, we do not need to worry about a maximum of $|f_s|$ on $\mc S^2_\pm$.

We apply the weighted Cauchy--Schwarz inequality $|ab|\leq\epsilon a^2+(1/4\epsilon)b^2$
to the term $8f^3g_s/g^5$ above, with $a=g_s/g$ and $b=f^3/g^4$.
Thus if $(f_s)_{\max}=C>0$ at some time, we obtain
\begin{align*}
\frac{\mr d}{\mr dt} (f_s)_{\max}
&\leq -(f_s)_{\max}\, \left(6 \frac{f^2}{g^4} + 2 \frac{g_s^2}{g^2}\right) + 8 \frac{f^3}{g^5} g_s \\
&\leq\left(\frac{4}{\sqrt3}-2C\right)\frac{g_s^2}{g^2}
+\left(\frac{12}{\sqrt3}\frac{f^4}{g^4}-6C\right)\frac{f^2}{g^4}\\
&\leq 0
\end{align*}
if $C \ge 2/\sqrt3$, because $f/g\leq1$.

A similar argument shows that $\frac{\mr d}{\mr dt} (f_s)_{\min}\geq0$ if $(f_s)_{\min}=-C$ at some time.
\end{proof}

These  uniform bounds on the first derivatives of $f$ and $g$ lead to the following.

\begin{lemma}
\label{lem-curv-1-der}
For any solution originating from initial data satisfying our Closeness Assumptions,
there exists a uniform constant $C$ such that 
\[
|\kappa_{12}| + |\kappa_{31}| + |\kappa_{23}| \le \frac{C}{g^2}
\]
for as long as the flow exists.
\end{lemma}

\begin{proof}
It follows from Lemma \ref{NearCalabi} that the inequality $(gg_s/f)^2\leq1$ persists if it is true initially.
This implies that $|g_s| \leq f/g$ for as long as the flow exists.
Combining this estimate with the identities in~\eqref{curvatures}, using Corollary~\ref{cor-gs}, Lemma~\ref{lem-fs},
and the fact that $f \leq g$, we obtain
\[
|\kappa_{12}| + |\kappa_{31}| + |\kappa_{23}| \le \frac{C}{g^2}.
\]
\end{proof}

\subsection{Second-derivative estimates}		\label{SecondOrder}

Here we derive estimates for the remaining curvatures ---
those that depend on second-order derivatives of $(f,g)$.

\begin{lemma}
\label{lem-k02}
For any solution originating from initial data satisfying our Closeness Assumptions,
there exists a uniform constant $C$ such that 
\[|\kappa_{02}| = \left|\frac{g_{ss}}{g}\right| \le \frac{C}{g^2}\]
for as long as the flow exists.
\end{lemma}

\begin{proof}
We define $Q = g g_{ss} - A g_s^2 - Bf_s^2$, where $A, B > 0$ are to be suitably chosen below.
We first show that  there exists a uniform constant $C$ so that $Q \ge -C$ for as long as the flow exists.
A straightforward computation shows that the evolution of $Q$ is given by
\begin{multline}
\label{eq-Q}
\frac{\partial Q}{\partial t} = \Delta Q + \frac{12 B f^2 f_s^2}{g^4}
+ \frac{4 f_s^2}{g^2} + \frac{24 f^2 g_s^2}{g^4} + \frac{12A f^2 g_s^2}{g^4} + \frac{2Af_s^2 g_s^2}{f^2} + \frac{4Bf_s^2g_s^2}{g^2} \\
+ \frac{2g_s^4}{g^2} + \frac{2A g_s^4}{g^2} + 2B f_{ss}^2 + 2(A - 1) g_{ss}^2 
- g g_{ss}\left(\frac{4f^2}{g^4} + \frac{2f_s^2}{f^2} + \frac{4Ag_s^2}{g^2}\right) \\
- \frac{16 B f^3 f_s g_s}{g^5} - \frac{24 f f_s g_s}{g^3} - \frac{8A f f_s g_s}{g^3} + \frac{2g f_s^3 g_s}{f^3} - \frac{8 g_s^2}{g^2} - \frac{8A g_s^2}{g^2}  \\
+ \frac{4 f f_{ss}}{g^2} + \frac{4B f_s^2 f_{ss}}{f} - \frac{8B f_s g_s f_{ss}}{g} - \frac{2g f_s g_s f_{ss}}{f^2},
\end{multline}
where as noted above, $\Delta Q=Q_{ss}+(f_s/f+2g_s/g)Q_s$.
We observe that  l'H\^opital's rule implies that the terms
\[
\frac{Q_s f_s}{f}, \qquad \frac{2Af_s^2 g_s^2}{f^2}, \qquad 2g f_s^2\, \frac{(f_s g_s - f g_{ss})}{f^3},
\qquad \frac{4B f_s^2 f_{ss}}{f},
\]
appearing in equation~\eqref{eq-Q} are well defined and smooth at $\mc S^2_{\pm}$.
We now distinguish between two cases.
\begin{case}
\label{case-1}
A minimum of $Q$ occurs away from $\mc S^2_\pm$.
\end{case}
We assume that at a minimum of $Q$ at some time $t$, we have $g g_{ss} - Ag_s^2 - Bf_s^2 \le -\bar{C}$
for a large constant $\bar{C}>0$ to be chosen. Because we are bounding $Q$ from below, we may assume that $g_{ss}\leq0$.
Then since Corollary~\ref{cor-gs} and Lemma~\ref{lem-fs} give uniform bounds for $|f_s|$ and $|g_s|$,
we may choose $\bar{C}$ sufficiently large relative to $A$ and $B$ such that
\begin{equation}
\label{eq-help-term}
-g g_{ss} \left(\frac{4f^2}{g^4} + \frac{2f_s^2}{f^2} + \frac{4Ag_s^2}{g^2}\right) \ge \frac{\bar{C} f^2}{2 g^4}
+ \frac{\bar{C} f_s^2}{2f^2} + \frac{\bar{C} A g_s^2}{g^2}.
\end{equation}
It then follows from \eqref{eq-Q} that at a minimum of $Q$  at time $t$, we have
\begin{equation}
\label{eq-Q100}
\begin{split}
\frac{\mr d}{\mr dt} Q_{\min} &\ge 2Bf_{ss}^2 +  \frac{\bar{C} f^2}{2 g^4}
+ \frac{\bar{C} f_s^2}{2f^2} + \frac{\bar{C} A g_s^2}{g^2} \\
&- \frac{16 B f^3 f_s g_s}{g^5} - \frac{24 f f_s g_s}{g^3} - \frac{8A f f_s g_s}{g^3} + \frac{2g f_s^3 g_s}{f^3}
- \frac{8 g_s^2}{g^2} - \frac{8A g_s^2}{g^2}\\ 
&+ \frac{4 f f_{ss}}{g^2} + \frac{4B f_s^2 f_{ss}}{f} - \frac{8B f_s g_s f_{ss}}{g} - \frac{2g f_s g_s f_{ss}}{f^2}.
\end{split}
\end{equation}
To estimate the terms in \eqref{eq-Q100} containing $f_{ss}$, we use Lemma~\ref{NearCalabi}, Corollary~\ref{cor-gs},
Lemma~\ref{lem-fs}, the facts that $f \le g$ and $|g_s|\leq f/g$, and a
weighted Cauchy--Schwarz inequality to determine  that there exists a uniform constant $C'$ such that
\begin{align*}
&\left|\frac{4 f f_{ss}}{g^2}\right| + \left|\frac{4B f_s^2 f_{ss}}{f}\right| + \left|\frac{8B f_s g_s f_{ss}}{g}\right|
+ \left|\frac{2g f_s g_s f_{ss}}{f^2}\right|\\
&\quad\leq\left(\frac 12 f_{ss}^2 + \frac{C' f^2}{g^4}\right) + \left(\frac{B}{2}\, f_{ss}^2 + C' B \frac{f_s^2}{f^2}\right)
+ \left(\frac B2 f_{ss}^2 + C'B \frac{g_s^2}{g^2}\right) + \left(\frac 12 f_{ss}^2 + C'\frac{f_s^2}{f^2}\right) \\
&\quad\leq(B+1) f_{ss}^2 + C'(B+1)\left( \frac{f_s^2}{f^2} + \frac{g_s^2}{g^2} +\frac{f^2}{g^4}\right).
\end{align*}
The remaining terms in~\eqref{eq-Q100} can be estimated in a similar manner. Thus we find that
\[
\begin{split}
\frac{\mr d}{\mr dt} Q_{\min} &\ge  2Bf_{ss}^2 +  \frac{\bar{C} f^2}{2 g^4} + \frac{\bar{C} f_s^2}{2f^2} + \frac{\bar{C} A g_s^2}{g^2} \\
&\qquad- C'(1 + A + B) \left(\frac{f^2}{g^4} + \frac{f_s^2}{f^2} + \frac{g_s^2}{g^2}\right) - (B + 1) f_{ss^2} \\
&\ge 0,
\end{split}
\]
if we choose $A = 1$, $B = 2$ and $\bar{C}$ sufficiently large so that $\bar{C} > C'(1 + A + B)$.
Therefore, in this case, either $Q \ge -\bar{C}$ or $\frac{\mr d}{\mr dt}Q_{\min} \ge 0$.

\begin{case}
\label{case-2}
A minimum of $Q$ occurs on $\mc S^2_\pm$.
\end{case}
The only difference from Case \ref{case-1} is that one must deal with the term
$\frac{Q_s f_s}{f}$ at $\mc S^2_\pm$. We apply l'H\^opital's rule to see that
\[
\frac{Q_s f_s}{f}\Big|_{\mc S^2_{\pm}}
= \left(Q_{ss} + Q_s \frac{f_{ss}}{f_s}\right)\Big|_{\mc S^2_{\pm}}
 = Q_{ss}\big|_{\mc S^2_{\pm}}.\]
However, smoothness of either function $Q_{\pm}(s,\cdot) :=Q(s - s_{\pm}, \cdot)$ at a minimum on $\mc S^2_\pm$
shows that $Q_{ss}|_{\mc S^2_{\pm}} = (Q_{\pm})_{ss}|_{\mc S^2_{\pm}} \ge 0$. A similar computation as in
Case~\ref{case-1} then yields
\[\frac{\mr d}{\mr dt}Q|_{\mc S^2_{\pm}} \ge 0,\]
unless $Q_{\min}(t) = Q(\cdot,t)|_{\mc S^2_{\pm}} \ge -\bar{C}$, for the constant $\bar{C}$ chosen in Case~\ref{case-1}.

Combined, Case~\ref{case-1} and Case~\ref{case-2} show that
\[
Q(\cdot,t) \ge \min\big\{-\bar{C},\,Q_{\min}(0)\big\}.
\]
In particular, this implies that
\[\frac{g_{ss}}{g} \ge -\frac{C}{g^2},\]
for a uniform constant $C$ as long as the flow exists.

Finally,  considering the quantity $\tilde{Q} := g g_{ss} + Ag_s^2 + Bf_s^2$ and bounding
$\tilde{Q}$ from above using similar arguments yields a uniform constant $C$ such that 
\[
\frac{g_{ss}}{g} \le \frac{C}{g^2}
\]
for as long as the flow exists. This concludes the proof of the Lemma.
\end{proof}

We now define
\begin{equation}	\label{define-lambda}
\mu(t) := \min_{\mc S^2\tilde\times\mc S^2} g(\cdot,t),
\end{equation}
observing that
Lemmas~\ref{lem-curv-1-der} and \ref{lem-k02} imply that there exists a uniform
constant $C$ such that as long as the flow exists, one has
\begin{equation}
\label{eq-help-111}
|\kappa_{12}| + |\kappa_{13}| + |\kappa_{23}| + |\kappa_{02}| + |\kappa_{03}| \le \frac{C}{\mu^2}.
\end{equation} 

\begin{remark}	\label{kappa-dichotomy}
Controlling the curvature $\kappa_{01}$ is considerably more subtle.
This is because, even for K\"ahler solutions, the alternative in statement \textsc{(ii)} of Lemma~\ref{lem-k01}
below is truly necessary: estimate~\eqref{eq-k01-above} need not hold unless such solutions
originate from initial data satisfying part~(d) of our Closeness Assumptions. Solutions for which part~(d)
is false can have $f\searrow0$ uniformly as $t\nearrow T$, with $g(\cdot,T)>0$ everywhere.
Each such (unrescaled) solution converges in the Gromov--Hausdorff sense to a
$\mb{CP}^1$ of multiplicity two; see Theorem~1.1 of \cite{SW11}.
\end{remark}

\begin{lemma}	\label{lem-k01}
For any solution originating from initial data satisfying our Closeness Assumptions, the following are true:\\
\textsc{(i)} The sectional curvature $\kappa_{01}=-f_{ss}/f$ satisfies
\begin{equation}
\label{eq-k01-below}
\kappa_{01} \ge -\frac{C}{g^2}
\end{equation}
for a uniform constant $C$.\\
\textsc{(ii)} Either there is an analogous upper bound
\begin{equation}
\label{eq-k01-above}
\kappa_{01} \le \frac{C}{\mu^2},
\end{equation}
or any finite-time singularity is Type-I.
\end{lemma}

\begin{proof}
Because the scalar curvature $R$ is a supersolution of the heat equation (in the sense that $(\partial_t - \Delta)\, R \ge 0$),
there exists a constant $r_0$ depending only on the initial data such that for as long as the flow exists, one has
\[r_0 \le R = \kappa_{01} + \kappa_{02} + \kappa_{03} + \kappa_{12} + \kappa_{23} + \kappa_{31},\]
where $\kappa_{02} = \kappa_{03}$. Using this together with Lemma~\ref{lem-curv-1-der}, Lemma~\ref{lem-k02},
and the fact that $\frac{\mr d}{\mr dt}g_{\max}\leq0$, we get the lower bound~\eqref{eq-k01-below}.

To prove \textsc{(ii)},  we assume that~\eqref{eq-k01-above} fails and use a blow-up argument
at a finite-time singularity.  In particular, we assume that $T < \infty$ is a singular time for the flow, and that
\begin{equation}
\label{eq-contra-ass}
\limsup_{t\to T} \left( \sup_{\mc S^2\tilde\times\mc S^2} \kappa_{01}(\cdot,t) \mu(t)^2\right) = \infty.
\end{equation}
We now let $t_i\to T$ as $i\to\infty$ such that
\[\
\sup_{t\in [0,t_i]} \left(\sup_{\mc S^2\tilde\times\mc S^2} \kappa_{01}(\cdot,t) \mu(t)^2\right) = \kappa_{01}(p_i,t_i) \mu(t_i)^2
\]
for some $p_i\in M$, and we let $K_i := \kappa_{01}(p_i,t_i)$.  It follows from our choice of $t_i$ that 
\begin{equation}
\label{eq-infty}
K_i \mu(t_i)^2 \to \infty\quad \mbox{as}\quad i\to\infty.
\end{equation}

We define the blow-up sequence of solutions $G_i$ of the metric of the form~\eqref{metric-standard} by
\[G_i(\cdot,t) := K_i \,G(\cdot, t_i + t K_i^{-1}),
\]
for $t$ satisfying
\[
-K_i t_i \le t < (T - t_i)K_i.
\]
We claim that the curvatures of the rescaled metrics $G_i$ are uniformly bounded.
To prove the claim for $\kappa_{12}$, say, we begin by noting that estimate~\eqref{eq-help-111} implies that
\begin{equation}
\label{eq-k12-ex}
\big|\kappa_{12}^i(\cdot,t)\big| = \frac{|\kappa_{12}(\cdot,t_i+tK_i^{-1})|}{K_i} \le \frac{C}{K_i \mu(t_i+tK_i^{-1})^2}.
\end{equation}
It follows from Remark~1 of \cite{IKS14} that the evolution equation for $g(\cdot,t)$ can be written as
\[\frac{\partial}{\partial t} (\log g) = -\kappa_{02} - \kappa_{23} - \kappa_{31},\]
which implies that
\[\left|\frac{\partial}{\partial t}\log g\right| \le \frac{C}{g^2},\]
and therefore that
\[\left|\frac{\mr d}{\mr dt}\mu^2\right| \le C.\]
Integrating this over $[t_i+tK_i^{-1}, t_i]$ yields
\[|\mu(t_i+tK_i^{-1})^2 - \mu(t_i)^2| \le \frac{C}{K_i},\]
for, say,  $t\in[-1,0]$. This implies that
\[\mu(t_i+tK_i^{-1})^2 \ge \mu(t_i)^2 - \frac{C}{K_i},\]
whereupon \eqref{eq-k12-ex} implies for $t\in [-1,0]$ that
\[
|\kappa_{12}^i(\cdot,t)| \le \frac{C}{K_i\mu(t_i)^2 - C} \to 0
\]
as $i\to \infty$, because \eqref{eq-infty} holds.

To bound the remaining curvatures of the rescaled metrics, we use similar arguments together with \eqref{eq-help-111} to conclude that 
\[
|\kappa_{12}^i(\cdot,t)| + |\kappa_{13}^i(\cdot,t)| + |\kappa_{23}^i(\cdot,t)|
+ |\kappa_{02}^i(\cdot,t)| + |\kappa^i_{03}(\cdot,t)| \le \frac{C}{K_i\mu(t_i)^2 - C} \rightarrow 0
\]
as $i\to \infty$, and we use  \eqref{eq-k01-below} to show that
\[\kappa_{01} \ge -\frac{C}{K_i\mu(t_i)^2 - C} \to 0,\]
as $i\to\infty$.

After extracting a convergent subsequence, we determine that $(\mc S^2\tilde\times\mc S^2, G_i(t), p_i)$ converges in the
pointed Cheeger--Gromov--Hamilton sense to a complete ancient solution
\[(\mc M^4_{\infty}, \mc{G}_{\infty}(t), p_{\infty})\]
that exists for $t\in (-\infty, t^*)$ where $t^*:= \lim_{i\to\infty} (T - t_i)\, K_i \le \infty$. 
Moreover, one has
\begin{equation}
\label{eq-kappa-infty}
\kappa_{12}^{\infty} = \kappa_{13}^{\infty} = \kappa_{23}^{\infty} = \kappa_{02}^{\infty} = 0,
\qquad\mbox{ and }\qquad\kappa_{01}^{\infty} \ge 0,
\end{equation}
with $\kappa_{01}(p_{\infty},0) = 1$. By applying Hamilton's splitting theorem  \cite{Ha1} twice,
we find that the universal cover $(\tilde{\mc M}^4_{\infty}, \mathcal{G}_{\infty},p_{\infty})$ splits isometrically
as the product of $\mathbb{R}^2$ and a complete  ancient solution $(\mc N^2, \mathcal{G}_{\infty}|_{\mc N^2})$
with bounded positive scalar curvature.
It follows from the classification in \cite{DHS} and \cite{DS}  that $(\mc N^2, \mc{G}_{\infty}|_{\mc N^2})$ is either
the King--Rosenau solution, the cigar, or the round sphere $\mc S^2$. In the former case, it is a standard
fact that by choosing a modified sequence $\tilde p_i$ of blow-up points, one can obtain the cigar as a limit.
But this is impossible by Perelman's $\kappa$-non-collapsing result~\cite{Pe}. So the limit must be isometric
to one of the products $\mc S^2\times\mb R^2$ or $\mb R^2\times\mc S^2$.
In either case,\footnote{For the metrics we study here,
the case $\mc S^2\times\mb R^2$ corresponds to the $g^2\big(\omega^2\ten\omega^2 + \omega^3\ten\omega^3\big)$
factor becoming flat after rescaling, while the case $\mb R^2\times\mc S^2$ corresponds to the
$\big(\ds\ten\ds+f^2\,\omega^1\ten\omega^1\big)$ factor becoming flat.}
the singularity is Type-I, and we have $\kappa_{01}\leq C/(T-t)$.
\end{proof}

\section{Singularity formation}	\label{Singular}
In this section, we investigate finite-time singularity formation for Ricci flow solutions originating from initial data
satisfying our Closeness Assumptions, with the objective --- not fully achieved --- of proving that
all such singularities are Type-I, with $|\mc S^2_-|=0$ at the singular time $T<\infty$.
This requires some work, for the following reason. Away from the special fibers $\mc S^2_\pm$, the geometry of
$(\mc S^2\tilde\times\mc S^2,G)$ is that of $(a,b)\times\mc S^3$. So without appropriate assumptions on the initial data,
it is highly plausible that neckpinch singularities like those analyzed in~\cite{IKS14} could develop at a
fiber $\{s_0\}\times\mc S^3$ far from $\mc S^2_-$. As explained below, we do not expect this possibility occurs
for solutions originating from initial data satisfying our Closeness Assumptions.
\smallskip

As proved in  \cite{SW11} and as noted above, the behavior of K\"ahler solutions depends strongly on
whether $|\mc S^2_+| < 3|\mc S^2_-|$, $|\mc S^2_+| = 3|\mc S^2_-|$, or $|\mc S^2_+| > 3|\mc S^2_-|$.
It follows from part~(d) of our Closeness  Assumptions that  the solutions we study have $|\mc S^2_+| > 3|\mc S^2_-|$ initially.
Our first result in this section proves that this threshold condition is preserved by the flow, even for
non-K\"ahler solutions, provided they originate from initial data satisfying the Closeness Assumptions.

\begin{lemma}	\label{big-end}
Solutions originating from initial data satisfying our Closeness Assumptions satisfy
 \[
 g^2(s_+,t) - 3g^2(s_-,t) \ge \delta^2
 \]
 for as long as they exist.
\end{lemma}

\begin{proof}
We recall that
\[g_t = g_{ss} + \left(\frac{f_s}{f} + \frac{g_s}{g}\right)\, g_s + 2\, \frac{(f^2 - 2g^2)}{g^3}.\]
Using l'H\^opital's rule, we compute at $s_+$ that
\[\lim_{s\to s_+} \frac{f_s g_s}{f} = g_{ss}(s_+,t).\]
Because $g_s(s_+,t) = 0$, we have 
\[\frac{\mr d}{\mr dt} g(s_+,t) = 2g_{ss}(s_+,t) - \frac{4}{g(s_+,t)}.\]
 Lemma \ref{NearCalabi} tells us that $g |g_s|\leq f$, which as a consequence of  l'H\^opital's rule, implies at $s_+$ that
\[g g_{ss} \ge -1.\]
It follows that
\begin{equation}
\label{eq-s+}
\frac{\mr d}{\mr dt} g^2(s_+,t) \ge -12.
\end{equation}

Similarly, using the fact that $g |g_s|\leq f$ and using l'H\^opital's rule, we have  that $g g_{ss} \le 1$ at $s_-$,  from which we obtain
\begin{equation}
\label{eq-s-}
\frac{\mr d}{\mr dt} g^2(s_-,t) \le -4.
\end{equation}
Estimates \eqref{eq-s+} and \eqref{eq-s-} together imply that
\[\frac{\mr d}{\mr dt} \big(g^2(s_+,t) - 3g^2(s_-,t)\big) \ge 0,\]
which yields
\[g^2(s_+,t) - 3g^2(s_-,t) \ge g^2(s_+,0) - 3g^2(s_-,0).\]
\end{proof}

Our second result in this section proves that solutions originating from initial data satisfying our Closeness
Assumptions become singular at $T<\infty$ only if $g$ vanishes somewhere.

\begin{lemma}
\label{lem-balancing}
If a solution originating from initial data satisfying our Closeness Assumptions becomes singular at time $T$,
then $\mu(T) = 0$.
\end{lemma}

\begin{proof}
Lemma~\ref{lem-k01} proves that either there is a two-sided curvature bound for $\kappa_{01}$
or the singularity is Type-I.

If there is a two-sided bound $|\kappa_{01}| \le C/\mu^2$, then combining this with 
estimate~\eqref{eq-help-111} we obtain a uniform constant $C$ such that
\[
|\Rc(G(t))| \le \frac{C}{\mu^2},
\]
for as long as the flow exists. Because \cite{Ses05} proves that $\limsup_{t\nearrow T}|\Rc|=\infty$
if $T<\infty$ is the singularity time, it follows that $\mu(T) = 0$.

To complete the proof, we may assume, to obtain a contradiction, that a solution encounters a finite-time Type-I
singularity for which $\lim_{t\to T} \mu(t) = 0$ is false.

We first claim that this assumption implies that there exists $\eta> 0$ such that $\mu(t) \ge \eta > 0$ for $t \in [0,T)$.
We prove this claim by contradiction.
Observe that the maximum principle implies that 
\[\frac{\mr d}{\mr dt} \mu(t) \ge -\frac{4}{\mu(t)}.\]
So for  $t \ge \tau$ in $[0,T)$, one has
\begin{equation}	\label{mu-estimate}
\mu(t)^2 \ge \mu(\tau)^2 - 8(t - \tau).
\end{equation}
If it is not true that $\lim_{t\to T} \mu(t) = 0$, then there exists a sequence $\tau_i \to T$ along which $\mu(\tau_i) \ge \eta > 0$ for all $i$.
On the other hand, if there exists another  sequence $t_i\to T$ along which $\lim_{i\to \infty} \mu(t_i) = 0$, then by passing to
subsequences, we may assume that $t_i \ge \tau_i$, and hence that
\[\mu(t_i)^2 \ge \mu(\tau_i)^2 - 8(t_i - \tau_i) \ge \eta^2 - 8(t_i - \tau_i).\]
But this is impossible, because $\lim_{i\to \infty}\mu(t_i) = 0$ and $\lim_{i\to\infty} (t_i - \tau_i) = 0$.
This contradiction proves the claim.

The proof of Lemma~\ref{lem-k01} tells us that the inequality $\mu(t) \ge \eta > 0$ implies that the universal cover of
any Type-I singularity model must be $\mc S^2\times\mathbb{R}^2$. Compactness of the $\mc S^2$
factor implies there is a sequence $t_i\to T$ along which $\sup_{s_-\leq s\leq s_+}f(s,t_i)\leq C\sqrt{T-t_i}$.
On the other hand it follows from Lemma \ref{NearCalabi} that 
\[g |g_s|\le f \le C\sqrt{T - t_i}\]
at those times, which implies that
\begin{equation}
\label{eq-est111}
g^2(s_+,t_i) - g^2(s_-,t_i) \le C\sqrt{T - t_i} (s_+ - s_-).
\end{equation}
We recall that
\[\frac{\mr d}{\mr dt}(s - s_-) = \int_{s_-}^s \left(\frac{f_{ss}}{f} + 2\frac{g_{ss}}{g}\right)\,\mr ds
= -\int_{s_-}^s (\kappa_{01} + 2\kappa_{02})\,\mr ds.\]
Combining  Lemma \ref{lem-k02} and part \textsc{(i)} of Lemma \ref{lem-k01},  we obtain
\[\frac{\mr d}{\mr dt} (s - s_-) \le \frac{C(s - s_-)}{\mu(t)^2} \le C^\prime(s - s_-),\] 
because $\mu(t) \ge \eta > 0$. Integrating this over $[0,T)$ yields a constant $C^{\prime\prime}$ such that
\begin{equation}
\label{eq-diam}
|s - s_-| \le C^{\prime\prime} \qquad\mbox{ for all }\quad t\in [0,T) \quad \mbox{ and }\quad s\in [s_-,s_+].
\end{equation}
Combining \eqref{eq-est111} and \eqref{eq-diam} then gives us
\[
\label{eq-contr1}
g^2(s_+,t_i) - g^2(s_-,t_i) \le C\, \sqrt{T - t_i}.
\]
But this is incompatible with the conclusion of Lemma \ref{big-end} that
\[g^2(s_+,t_i) - 3g^2(s_-,t_i) \ge \delta^2>0.\]
This contradiction proves the result.
\end{proof}

For K\"ahler solutions, monotonicity of $g$ is preserved automatically for as long as the metric remains smooth.
We do not know if this is true for the non-K\"ahler solutions studied here. However, it
follows from the evolution equation for $g_s$,
\[
 (g_s)_t = \Delta(g_s)-2\frac{g_s}{g}(g_s)_s
 +\left\{\frac{4}{g^2} - \frac{g_s^2}{g^2} - \frac{f_s^2}{f^2} - 6\frac{f^2}{g^4}\right\} g_s
 + 4\frac{f}{g^3} f_s,
 \]
 that monotonicity can fail only where $f_s<0$, \emph{i.e.,} only in a proper neighborhood of $\mc S^2_+$.
 In our construction of initial data in Section~\ref{sec-initial}, we are free to choose the parameter
 $\alpha^2=|\mc S^2_-|$ as small as possible, and the parameter $A$, which controls the size of
$\mc S^2_+$ up to an $\ve$ error, as large as possible.  Moreover, estimate~\eqref{eq-s-} shows that
\[
\frac{\mr d}{\mr dt} g^2(s_-,t) \le -4,
\]
while at an interior minimum $s_{\mr{neck}}$ of $g$, it is easy to see that
\[
(g^2)_t\big|_{s=s_{\mr{neck}}}\geq-8.
\]
This line of reasoning strongly suggests that  it should be possible to construct an open set of initial data for
which $g^2$ vanishes at $s_-$ before it can vanish at an interior point. What keeps this formal argument
from being a rigorous proof is that in order to obtain a uniform lower bound for $g$ in the neighborhood where a local
minimum can form, one needs a uniform bound from below on the distance between $\mc S^2_-$ and the first
critical point of $f$. However, it is notoriously difficult to control the location of a critical point of a solution of a parabolic
\textsc{pde}. Nevertheless, we believe the following to be true:

\begin{conjecture}	\label{SouthPole}
For appropriate choices of $\alpha\ll A$, a solution originating from initial data satisfying our Closeness Assumptions
satisfies
\[
\mu(t)=g(s_-,t)
\]
for as long as it exists.
\end{conjecture}

We now proceed under the \textbf{assumption} that Conjecture~\ref{SouthPole} is true.
If so, then recalling Lemma~\ref{lem-balancing}, ones sees that solutions originating from initial data satisfying
our Closeness Assumptions become singular only by crushing the fiber $\mc S^2_-$. We state this as follows:

\begin{corollary}	\label{crush}
If a solution originating from initial data satisfying our Closeness Assumptions becomes singular at time $T$,
then $g(s_-,T) = 0$.
\end{corollary}

\begin{corollary}	\label{TypeOne}
All solutions originating from initial data satisfying our Closeness Assumptions develop finite-time Type-I singularities.
\end{corollary}

\begin{proof}
Assuming Conjecture~\ref{SouthPole}, it follows from estimate~\eqref{eq-s-} that
\begin{equation}	\label{crush-est}
\frac{\mr d}{\mr dt}\big(\mu^2(t)\big)\leq-4.
\end{equation}
So a finite-time singularity is inevitable.
As a consequence of  Lemma~\ref{lem-k01}, to prove that the singularity  is Type-I, we may assume there is a two-sided curvature bound for $\kappa_{01}$.
Such a bound, together with estimate~\eqref{eq-help-111}, gives a uniform constant $C$ such that
$|\Rc(G(t))| \leq C\mu^{-2}(t)$ for as long as the flow exists. But then the result follows easily from estimate~\eqref{crush-est}.
\end{proof}

\section{Convergence to the blowdown soliton}

Corollaries~\ref{crush} and \ref{TypeOne} tell us that any point $p\in\mc S^2_-$ is a \emph{special Type-I singular point} in the sense
of Enders--M\"uller--Topping~\cite{EMT}. It follows from that work that  every blow-up sequence $\big(\mc S^2\tilde\times\mc S^2,G_k(t),p\big)$
subconverges to a smooth nontrivial gradient shrinking soliton $\big(\mc M, G_\infty(\tau)\big)$ defined for $-\infty<\tau<0$.
Using  Lemma~\ref{big-end}, we determine that the limit is noncompact.
So $\mc M$ is diffeomorphic to $\mb C^2$ blown up at the origin; that is, $\mc O(-1)$.
Moreover, the symmetries of $G(t)$ are preserved in the limit, so the metric retains  the form exhibited  in~\eqref{metric-standard}:
\begin{equation}	\label{limit}
G_\infty = \ds\ten\ds+\Big\{f^2\,\omega^1\ten\omega^1
 	   + g^2\big(\omega^2\ten\omega^2
	   + \omega^3\ten\omega^3\big)\Big\}.
\end{equation}
Here and in the remainder of this section, we abuse notation by using $s$ to represent arclength
from $\mc S^2_-$ in the limit soliton, and using $f$ and $g$ for the other components of the limit soliton metric.

The quantity $\psi$ that we estimate in Lemma~\ref{NearCalabi} is scale-invariant, so the limit soliton satisfies
\begin{equation}	\label{almostKahler}
-1\leq\frac{gg_s}{f}\leq1,
\end{equation}
which implies that the limit is ``not too far'' from K\"ahler in a precise sense. It is a general principle that shrinking solitons
appear in discrete rather than continuous families, modulo scaling and isometry. So it is reasonable to expect that there are no other
cohomogeneity-one shrinking solitons in the neighborhood of such metrics satisfying estimate~\eqref{almostKahler}.
(For a related rigidity result, see work~\cite{Kot17} of Kotschwar.) To obtain this result, however,
we require another assumption.

We now introduce that second conjecture and then present the formal argument that motivates us to
believe it is true:

\begin{conjecture} 	\label{Kahler}
If $\big(\mc M, G_\infty(\tau)\big)$ is a smooth gradient shrinking soliton 
having the form~\eqref{limit} obtained as a limit of parabolic rescalings of a solution 
originating from initial data that satisfy our  Closeness Assumptions for suitable $\alpha\ll A$,
then
\[
gg_{ss}\big|_{\mc S^2_-}=1.
\]
\end{conjecture}

Since by l'H\^opital's rule, $gg_{ss}\big|_{\mc S^2_-}=\lim_{s\searrow s_-} (gg_s/f)$, we call this an
``infinitesimal K\"ahler condition''.
Our formal argument that it should hold on the limit soliton is based on a parabolic rescaling of the original solution in a
neighborhood of the developing singularity on $\mc S^2_-$. For clarity in the argument,  we write
$\zeta_t\big|_\xi$ to indicate that we are taking the time
derivative of a smooth space-time function $\zeta$ with a spatial variable $\xi$ held fixed. 
All time derivatives computed thus far have been derived from~\eqref{RF}, in which $x$ is held fixed.

Using \eqref{dsdr} along with equations~\eqref{u-evolution} and \eqref{v-evolution}, one
computes that the evolution equations for $u$ and $v$ with $x$ held fixed may be written
with respect to the $\varrho$ variable as
\begin{subequations}		\label{FirstStep}
\begin{align}
\frac14\,u_t\big|_x&=\frac{u_{\varrho\varrho}}{u}-\frac{u_\varrho^2}{u^2}+\frac{u_\varrho v_\varrho}{uv}-\frac{u^2}{v^2},\\
\frac14\,v_t\big|_x&=\frac{v_{\varrho\varrho}}{u}+\frac{u}{v}-2,		\label{KRF-v}
\end{align}
\end{subequations}
respectively. Motivated by Corollary~\ref{TypeOne} and the geometry of the blowdown soliton, we introduce
new time and space variables,
\[
\tau:=-\log\big\{4(T-t)\big\}\qquad\mbox{ and }\qquad\sigma:=\sqrt2\,\tau+\varrho,
\]
where $T<\infty$ is the singularity time. We then define rescaled metric components,
\[
U(\sigma,\tau):=e^\tau u(s,t)\qquad\mbox{ and }\qquad V(\sigma,\tau):=e^\tau v(s,t),
\]
noting that a solution is K\"ahler if and only if $U=V_\sigma$.
We observe that equation~\eqref{rho-evolution} implies that
\[
\sigma_\tau=\sqrt2+\mc J,
\]
where $\mc J:=\varrho_\tau=e^{-\tau}\varrho_t/4$. 
To compute the nonlocal, nonlinear term $\mc J$, we note that 
\[
f_s =\frac{U_{\sigma}}{U}, \qquad g_s = \frac{V_{\sigma}}{\sqrt{UV}},
\]
and
\[
g_{ss} = e^{\tau/2}\left\{\frac{2V_{\sigma\sigma}}{UV^{1/2}}
-\frac{V_{\sigma}\,(UV)_\sigma}{U^2 V^{3/2}}\right\}.
\]
Applying these transformations to formula~\eqref{rho-evolution} shows that
in these coordinates, the nonlocal term is given by
\begin{equation}	\label{J-in-sigma}
\mc J = \int_{\sqrt{2}\tau}^{\sigma}
\left\{\frac{V_{\bar\sigma\bar\sigma}}{U V} - \frac{U_{\bar\sigma}\, V_{\bar\sigma}}{U^2 V}
-\frac12\frac{V_{\bar\sigma}^2}{U V^2}+\frac12 \frac{U}{V^2}\right\}\,\mr d\bar{\sigma}.
\end{equation}
The conversion from time derivatives with $x$ held fixed to  time derivatives with $\sigma$ held fixed
is given by
\begin{align*}
U_\tau\big|_\sigma+\big(\sqrt2+\mc J\big)U_\sigma-U&=\frac14u_t\big|_x,\\
V_\tau\big|_\sigma+\big(\sqrt2+\mc J\big)V_\sigma-V&=\frac14v_t\big|_x.
\end{align*}
Thus by using~\eqref{FirstStep}, we obtain the evolution equations
\begin{subequations}
\begin{align}
U_\tau\big|_\sigma&=\frac{U_{\sigma\sigma}}{U}-\big(\sqrt2+\mc J\big)U_\sigma
-\frac{U_\sigma^2}{U^2}+\frac{U_\sigma V_\sigma}{UV}-\frac{U^2}{V^2}+U,
\label{U-evolution}\\
\notag\\
V_\tau\big|_\sigma&=\frac{V_{\sigma\sigma}}{U}-\big(\sqrt2+\mc J\big)V_\sigma
+\frac{U}{V}+V-2.	\label{V-evolution}
\end{align}
\end{subequations}

\begin{remark}
By comparing equations~\eqref{BlowdownEvolution} and \eqref{V-evolution}, one finds that if $\Phi$
is the rescaling of the blowdown soliton $\vp$, then $V=\Phi$ evolves by $V_\tau\big|_\sigma=-\mc J\,V_\sigma$.
But by Lemma~\ref{rho-evolution-proof}, $\mc J$ vanishes on any K\"ahler solution. Hence $V=\Phi$
becomes a stationary solution in these coordinates.
\end{remark}

Motivated by the quantity $\psi$ introduced in~\eqref{KahlerQuantity}, we now define
\[
\Omega:=\frac{V_\sigma}{U}.
\]
Then differentiating equations~\eqref{J-in-sigma} and \eqref{V-evolution},
recalling \eqref{U-evolution}, and arranging terms, one computes that $\Omega$ evolves by
\begin{equation}	\label{W-evolution}
\Omega_\tau\big|_\sigma
=\frac{\Omega_{\sigma\sigma}}{U}
+\left\{\frac{U_\sigma}{U^2}-\frac{\Omega}{V}-\sqrt2-\mc J\right\}\Omega_\sigma
+ (1 - \Omega^2)\, \left( \frac{U_{\sigma}}{U V} - \frac12\frac{U \Omega}{V^2}\right).
\end{equation}
We note that $\Omega\equiv1$ is a stationary solution, which reflects the fact that
the K\"ahler condition is preserved under Riemannian Ricci flow.

To linearize, we define $\omega:=\Omega-1$ and compute that
\begin{equation}	\label{omega-evolution}
\omega_\tau\big|_\sigma= \frac{\omega_{\sigma\sigma}}{U}
+\left\{\frac{U_\sigma}{U^2}-\frac{1}{V}-\sqrt2\right\}\omega_\sigma
+\left\{\frac{U}{V^2}-2\frac{U_\sigma}{UV}\right\}\omega+Q[\omega],
\end{equation}
where the nonlinear terms on the \textsc{rhs} are given by
\[
Q[\omega]=-\left\{\frac{\omega}{V}+\mc J\right\}\omega_\sigma
+\left\{\frac12\frac{U\left(3+\omega\right)}{V^2}-\frac{U_\sigma}{UV}\right\}\omega^2.
\]
Here we use the fact that
\[
\mc J=\int_{\sqrt2\tau}^\sigma
\left\{\frac{\omega_{\bar\sigma}}{UV}-\frac{U(\omega+\omega^2/2)}{V^2}\right\}
\mr d\bar\sigma.
\]

If $\omega$ is small, we are close to a K\"ahler solution. It then follows
from~\eqref{ClosingCondition} that $V=1+a_1e^\sigma+a_2e^{2\sigma}+\cdots$ and
$U=a_1e^\sigma+2a_2e^{2\sigma}+\cdots$.
Thus as $\sigma\searrow-\infty$, \emph{i.e.,} in a neighborhood of $\mc S^2_-$, the factor multiplying
$\omega$ in the linear reaction term of equation~\eqref{omega-evolution} satisfies
\[
\frac{U}{V^2}-2\frac{U_\sigma}{UV}\approx-2,
\]
which leads us to expect ``asymptotic approach to K\"ahler'' in that neighborhood, and thus motivates us to make Conjecture~\ref{Kahler}.
\smallskip

\textbf{Assuming} Conjecture~\ref{SouthPole} and Conjecture~\ref{Kahler} are true, we now prove the following:

\begin{lemma}
Any smooth gradient shrinking soliton $\big(\mc M, G_\infty(\tau)\big)$
having the form~\eqref{limit} obtained as a limit of parabolic rescalings at $\mc S^2_-$ is K\"ahler.
\end{lemma}

\begin{proof}
We work at a fixed time $\tau<0$ and so suppress time below. However, we continue to use subscripts to indicate
spatial derivatives. We note here that smoothness requires that the closing conditions~\eqref{close-} hold at $s=0$,
a fact we use freely below.
\smallskip

We define
\[
F(s)=f-gg_s.
\]
We have $F(0)=0$ by smoothness of the metric, and $F_s(0)=0$ by Conjecture~\ref{Kahler},
because $F_s(0)=1-gg_{ss}$. We proceed to show that $F=0$ for all $s$. 

We  denote the soliton potential function by $\Gamma$  and we set $\gamma=\Gamma_s$.
Using equation~(51) from~\cite{IKS14}
to compute the Lie derivative, we find that the soliton equation
\begin{equation}
-\Rc[G_\infty]=\lambda G_\infty+\frac12\mc L_{\nabla \Gamma}G_\infty
\end{equation}
becomes the system
\begin{subequations}		\label{soliton}
\begin{align}
\gamma_s&=\frac{f_{ss}}{f}+2\frac{g_{ss}}{g}-\lambda,	\label{h1d} \\
\frac{f_{ss}}{f}&=\frac{f_s\gamma}{f}-2\frac{f_sg_s}{fg}+2\frac{f^2}{g^4}+\lambda,	\label{f2d} \\
\frac{g_{ss}}{g}&=\frac{g_s\gamma}{g}-\frac{f_sg_s}{fg}-\frac{g_s^2}{g^2}-2\frac{f^2}{g^4}+\frac{4}{g^2}+\lambda, \label{g2d}
\end{align}
\end{subequations}
where $\lambda<0$ depends on our choice of $\tau$ above.

Computing $F_s$ using equation~\eqref{g2d}, one finds that
\begin{align}
F_s&=f_s-g_s^2-gg_{ss}\notag \\
&=f_s-gg_s\gamma+\frac{gf_sg_s}{f}+2\frac{f^2}{g^2}-\lambda g^2-4\notag\\
&=\left(\gamma-\frac{f_s}{f}\right)F+2f_s-f\gamma+2\frac{f^2}{g^2}-\lambda g^2 -4. \label{F1d}
\end{align}
Hence
\[
F_{ss}=\left(\gamma-\frac{f_s}{f}\right)F_s+\left(\gamma-\frac{f_s}{f}\right)_sF+X,
\]
where we use~\eqref{f2d} to rewrite the final term above as
\begin{align*}
X&=2f_{ss}-f_s\gamma-f\gamma_s+4\left(\frac{ff_s}{g^2}-\frac{f^2g_s}{g^3}\right)-2\lambda gg_s\\
&=\left(4\frac{f^2}{g^4}+4\frac{f_s}{g^2}+2\lambda\right)F+\gamma^2\left(\frac{f}{\gamma}\right)_s.
\end{align*}
Therefore, $F$ satisfies the linear second-order (seemingly inhomogeneous) \textsc{ode}
\begin{equation}
F_{ss}-\left(\gamma-\frac{f_s}{f}\right)F_s
-\left\{\left(\gamma-\frac{f_s}{f}\right)_s+4\frac{f_s}{g^2}+4\frac{f^2}{g^4}+2\lambda\right\} F=
\gamma^2\left(\frac{f}{\gamma}\right)_s.	\label{LiH}
\end{equation}

We now show that the term on the \textsc{rhs} can be rewritten in terms of $F$ and $F_s$.
Using equations~\eqref{h1d}, \eqref{f2d}, and \eqref{g2d} in order, and then applying the identity $gg_s=f-F$, we obtain
\begin{align}
\frac12 \gamma^2\left(\frac{f}{\gamma}\right)_s&=\frac12(f_s\gamma-f\gamma_s)\notag\\
&=\frac12\left(f_s\gamma-f_{ss}-2\frac{fg_{ss}}{g}+\lambda f\right)\notag\\
&=\frac{f_sg_s}{g}-\frac{f^3}{g^4}-\frac{fg_{ss}}{g}\notag\\
&=2\frac{f_sg_s}{g}-\frac{fg_s\gamma}{g}+\frac{fg_s^2}{g^2}+\frac{f^3}{g^4}-4\frac{f}{g^2}-\lambda f\notag\\
&=-\left(2\frac{f_s}{g^2}+\frac{f^2}{g^4}+\frac{fg_s}{g^3}-\frac{f\gamma}{g^2}\right)F+Y,	\label{inhom}
\end{align}
where
\[
Y=2\frac{ff_s}{g^2}+2\frac{f^3}{g^4}-4\frac{f}{g^2}-\frac{f^2\gamma}{g^2}-\lambda f.
\]
Using equation~\eqref{F1d} to rewrite the first term on the \textsc{rhs}, it is easy to see that
\begin{equation}
Y=\frac{f}{g^2}F_s+\frac{f_s-f\gamma}{g^2}F.	\label{Ynice}
\end{equation}
So by using  equations~\eqref{inhom} and \eqref{Ynice}, we find that equation~\eqref{LiH} can be rewritten as
the linear second-order \emph{homogeneous} \textsc{ode}
\[
F_{ss}+\left(\frac{f_s}{f}-\gamma-2\frac{f}{g^2}\right)F_s
+\left\{\left(\frac{f_s}{f}-\gamma\right)_s+2\frac{fg_s}{g^3}-2\frac{f_s}{g^2}-2\frac{f^2}{g^4}-2\lambda\right\} F=0.
\]
Because $f_s/f\sim1/s$ and $(f_s/f)_s\sim-1/s^2$ as $s\searrow0$, this \textsc{ode} has a regular singular point at $s=0$. 
It is approximated in a neighborhood of $s=0$ by the equidimensional Euler equation
$s^2y''(s)-sy'(s)+y(s)=0$,
for which a fundamental set of solutions is $\{s,\,s\log(s))\}$.
It then follows from a theorem of Frobenius that a fundamental set of solutions of the exact equation has the form
\[
\sum_{n=0}^\infty a_ns^{n+1}\qquad\mbox{ and }\qquad \sum_{n=0}^\infty b_ns^n\log(s),
\]
where all coefficients except $a_0$ and $b_0$ are determined by recurrence relations. We conclude that
$F$ is identically zero for all $s\geq0$, hence that the soliton is K\"ahler.
\end{proof}
Theorem~1.5 of \cite{FIK03} tells us that the blowdown soliton is unique up to scaling and
isometry among $\mr U(2)$-invariant K\"ahler--Ricci solitons.
Hence this completes our presentation of evidence in favor of our Main Conjecture.

\end{document}